\documentclass[12pt,reqno]{amsart}

\usepackage[margin=1in]{geometry}
\usepackage{amsmath,amssymb,amsfonts,amsthm}
\usepackage{mathrsfs}
\usepackage[utf8]{inputenc}
\usepackage[T2A,T1]{fontenc}
\usepackage[english]{babel}
\usepackage{csquotes}
\usepackage{lmodern}
\usepackage{hyperref}
\usepackage{tikz-cd}
\usepackage[final]{microtype}

\numberwithin{equation}{section}

\usepackage[
  backend=biber,
  style=alphabetic, 
  sorting=nyt, 
  giveninits=true, 
  maxbibnames=99 
]{biblatex}
\renewbibmacro{in:}{}

\newtheorem{theorem}{Theorem}[section]
\newtheorem{conjecture}[theorem]{Conjecture}
\newtheorem{corollary}[theorem]{Corollary}
\newtheorem{lemma}[theorem]{Lemma}
\newtheorem{proposition}[theorem]{Proposition}
\newtheorem{situation}[theorem]{Situation}
\theoremstyle{definition}
\newtheorem{remark}[theorem]{Remark}

\newcommand{\Z}{\mathbb{Z}}
\newcommand{\Q}{\mathbb{Q}}
\newcommand{\CC}{\mathbb{C}}
\newcommand{\pri}{\mathfrak{p}}
\newcommand{\XX}{\mathcal{X}}

\newcommand{\Rcal}{\mathcal{R}}
\newcommand{\inject}{\hookrightarrow}
\newcommand{\onto}{\twoheadrightarrow}
\newcommand{\et}{\mathrm{\acute{e}t}}
\newcommand{\dR}{\mathrm{dR}}
\newcommand{\loc}{\mathrm{loc}}
\newcommand{\proj}{\mathbf{P}}
\newcommand{\aff}{\mathbf{A}}
\newcommand{\OO}{\mathcal{O}}
\newcommand{\VV}{\mathcal{V}}
\newcommand{\Hs}{\mathscr{H}}
\newcommand{\UU}{\mathcal{U}}

\DeclareSymbolFont{cyrillic}{T2A}{cmr}{m}{n}
\DeclareMathSymbol{\Sha}{\mathalpha}{cyrillic}{216}

\DeclareMathOperator{\GL}{GL}
\DeclareMathOperator{\Res}{Res}
\DeclareMathOperator{\Gal}{Gal}
\DeclareMathOperator{\Alb}{Alb}

\DeclareMathOperator{\Lie}{Lie}
\DeclareMathOperator{\rank}{rank}
\DeclareMathOperator{\pr}{pr}
\DeclareMathOperator{\codim}{codim}
\DeclareMathOperator{\Hod}{Hod}
\DeclareMathOperator{\MT}{\mathbf{MT}}

\title{Functional transcendence for the unipotent Albanese map}
\author{Daniel Rayor Hast}

\begin{document}

\begin{abstract}
  We prove a certain transcendence property of the unipotent Albanese map of a smooth variety, conditional on the Ax--Schanuel conjecture for variations of mixed Hodge structure. We show that this property allows the Chabauty--Kim method to be generalized to higher-dimensional varieties. In particular, we conditionally generalize several of the main Diophantine finiteness results in Chabauty--Kim theory to arbitrary number fields.
\end{abstract}

\maketitle

\section{Introduction}
Let $X$ be a smooth, connected, positive-dimensional algebraic variety of good reduction over a nonarchimedean local field $K/\Q_p$, and let $\XX/\OO_K$ be a smooth integral model. If $X$ is defined over a number field $F \subset K$, and $S$ is a finite set of primes of $F$ such that $\XX$ is defined over the ring of $S$-integers $\OO_{F, S}$, a fundamental problem in arithmetic geometry is to study the set of $S$-integral points $\XX(\OO_{F, S})$. In particular, when $X$ is a hyperbolic curve, $\XX(\OO_{F, S})$ is finite by Faltings' theorem \cite{Faltings83,Faltings84}, and there are major open problems around effective computation of $\XX(\OO_{F, S})$ and optimal or uniform bounds on the size of $\XX(\OO_{F, S})$. More generally, when $X$ is a variety of general type, Lang conjectured \cite[Conj.\ 5.7]{Lang86} that $\XX(\OO_{F, S})$ is non-Zariski-dense in $X$.

This paper is concerned with one approach to studying integral or rational points on varieties, the non-abelian Chabauty method. Most prior work in this method has placed substantial restrictions on the base field (in many cases, limited to $\Q$ only). In this paper, we show how a Hodge-theoretic conjecture of Klingler can be used to generalize the non-abelian Chabauty method to higher-dimensional varieties, and hence, via restriction of scalars, to curves over arbitrary number fields. The non-abelian Chabauty method produces $p$-adic analytic functions on which the set of integral points vanishes; the key to proving finiteness or non-Zariski-density results in higher dimensions is to show that these functions are independent of each other in a suitable sense.

We begin by briefly summarizing how the non-abelian Chabauty method works. Fix a base point $b \in \XX(\OO_K)$. Let $\Pi^{\dR}$ be the de Rham fundamental group of $X$ with base point $b$; denote the lower central series of $\Pi^{\dR}$ by $(\Pi^{\dR})^{n + 1} := [\Pi^{\dR}, (\Pi^{\dR})^n]$, and let $\Pi_n^{\dR} := \Pi^{\dR}/(\Pi^{\dR})^{n + 1}$. This group $\Pi_n^{\dR}$ is a unipotent $K$-algebraic group with Hodge filtration $F^\bullet \Pi_n^{\dR}$ and a Frobenius action induced by comparison with the crystalline fundamental group.

Kim \cite{Kim-alb} proves that $\Pi_n^{\dR}/F^0$ is a moduli space for \emph{admissible torsors} of $\Pi_n^{\dR}$, i.e., torsors for $\Pi_n^{\dR}$ with Hodge filtration and Frobenius action that are trivializable separately for the Hodge filtration and the Frobenius action (the trivial torsor corresponding to the case where the Hodge filtration and Frobenius action are compatible). Kim then constructs the \emph{unipotent Albanese map}, a rigid $K$-analytic map
\[
  j_n\colon \XX(\OO_K) \to (\Pi_n^{\dR}/F^0)(K)
\]
defined by sending each $x \in \XX(\OO_K)$ to the class of the path torsor $P_{b, x}^{\dR}$ defined via the de Rham fundamental groupoid. Kim proves \cite[Thm.\ 1]{Kim-alb} that the image of $j_n$ is Zariski-dense, and uses this to formulate a non-abelian generalization of Chabauty's method. The reason Zariski-density is useful is that, if $Z \subseteq \Pi_n^{\dR}/F^0$ is an algebraic subvariety of positive codimension, then $j_n^{-1}(Z)$ is a $K$-analytic subvariety of lower dimension than $\dim X$ (and in particular, $j_n^{-1}(Z)$ is finite if $X$ is a curve).

(In applications to integral or rational points over a number field $F$, the subvariety $Z$ consists of the torsors that arise from a global torsor via the localization map $\log_p = D \circ \loc_p$, defined as the composition
\[
\begin{tikzcd}
  H_f^1(\Gal(\bar{F}/F), \Pi_n^{\et}) \ar[r, "\loc_p"] & H_f^1(\Gal(\bar{F}_v/F_v), \Pi_n^{\et}) \ar[r, "D"] & \Res_{\Q_p}^{F_v}(\Pi_n^{\dR}/F^0),
\end{tikzcd}
\]
where $\loc_p$ is given by restricting the Galois action to the local Galois group, and $D$ is a comparison map arising from $p$-adic Hodge theory. The set of integral points factors through this subvariety via a global \'etale realization of the unipotent Albanese map. However, for the purposes of this paper, the origin of the subvariety $Z$ is irrelevant.)

In this paper, we will see that the unipotent Albanese maps conditionally satisfy a much stronger transcendence property, a $K$-analytic version of the Ax--Schanuel conjecture. From this, we deduce that, for $Z$ of high enough codimension, $j_n^{-1}(Z)$ is non-Zariski-dense (not just non-dense in the $K$-analytic topology). For simplicity, we only work on the residue disk $\UU \subseteq \XX(\OO_K)$ containing the base point $b$ (which is harmless for finiteness results since $\XX(\OO_K)$ is a finite union of residue disks); when we refer to an algebraic subvariety of a residue disk, we mean the intersection of an algebraic subvariety of $X$ with the residue disk.

\begin{theorem}
  \label{thm:main}
  Let $X$ be a smooth, connected, positive-dimensional algebraic variety of good reduction over a nonarchimedean local field $K/\Q_p$, and let $\XX/\OO_K$ be a smooth integral model. Let $n \geq 1$, and assume the Ax--Schanuel conjecture for the $n$-th canonical unipotent variations of mixed Hodge structure on $X(\CC)$ (under some fixed embedding $K \inject \CC$). Let $V \subseteq X \times \Pi_n^{\dR}/F^0$ be an algebraic subvariety. Let $\Gamma \subseteq X \times (\Pi_n^{\dR}/F^0)(K)$ be the graph of $j_n$. Let $W$ be an irreducible analytic component of $V \cap \Gamma \cap \UU$. Then
  \[
    \codim_{X \times \Pi_n^{\dR}/F^0}(W) \geq \codim_{X \times \Pi_n^{\dR}/F^0}(V) + \codim_{X \times \Pi_n^{\dR}/F^0}(\Gamma)
  \]
  unless $\pr_X(W)$ is contained in a proper weakly special subvariety of $X$ (in the sense of \cite[Definition 7.1]{Klingler17}) with respect to the $n$-th canonical unipotent variation of mixed Hodge structure.
\end{theorem}

(Note: in the case $n = 1$, the Ax--Schanuel conjecture is just the Ax--Schanuel theorem for abelian varieties \cite[Thm.~1]{Ax72}, and the weakly special subvarieties are translates of abelian subvarieties. In particular, Theorem \ref{thm:main} and Corollary \ref{cor:main} are unconditional for $n = 1$.)

\begin{corollary}
  \label{cor:main}
  In the setting of Theorem \ref{thm:main}, let $Z \subseteq \Pi_n^{\dR}/F^0$ be a closed algebraic subvariety of codimension at least $\dim X$. Then there is a finite subset $S \subseteq \UU$ such that $\UU \cap j_n^{-1}(Z) \setminus S$ is contained in a finite union of proper weakly special subvarieties of $X$.
\end{corollary}
\begin{proof}
  The proof closely follows \cite[Cor.~9.2]{Lawrence-Venkatesh}. Let $V = X \times Z$. Then $V \cap \Gamma \cap \UU = (\UU \cap j_n^{-1}(Z)) \times Z$. Since $Z$ has only finitely many irreducible components, $\UU \cap j_n^{-1}(Z) = \pr_X(V \cap \Gamma \cap \UU)$ also has only finitely many irreducible analytic components.

  Let $W$ be an irreducible analytic component of $V \cap \Gamma \cap \UU$. Applying Theorem \ref{thm:main}, either $\pr_X(W)$ is contained in a proper weakly special subvariety of $X$, or
  \begin{align*}
    \codim_{X \times \Pi_n^{\dR}/F^0}(W) &\geq \codim_{X \times \Pi_n^{\dR}/F^0}(V) + \codim_{X \times \Pi_n^{\dR}/F^0}(\Gamma) \\
    &= \codim_X(Z) + \dim(\Pi_n^{\dR}/F^0) \geq \dim X + \dim(\Pi_n^{\dR}/F^0).
  \end{align*}
  In the latter case, $W$ must be zero-dimensional.

  Let $S$ be the (finite) union of all of the zero-dimensional irreducible analytic components of $V \cap \Gamma \cap \UU$. Then $\UU \cap j_n^{-1}(Z) \setminus S$ is a finite union of sets that are each contained in a proper weakly special subvariety of $X$, completing the proof.
\end{proof}

The key input is the Ax--Schanuel conjecture for variations of mixed Hodge structure.

\begin{conjecture}[Ax--Schanuel for variations of mixed Hodge structure]
  \label{conj:ax-schanuel}
  Let $X$ be a smooth connected algebraic variety over $\CC$, and fix a base point $b \in X(\CC)$. Let $\Hs_\Z$ be a $\Z$-variation of mixed Hodge structure on $X$ with generic Mumford--Tate group $\MT_{\Hs_\Z}$. Let $\Lambda \subseteq \MT_{\Hs_\Z}(\Z)$ be the image of the monodromy representation $\pi_1(X, b) \to \MT_{\Hs_\Z}(\Z)$ (after passing to a finite cover of $X$ if necessary). Let $G$ be the $\Q$-Zariski closure of $\Lambda$ in $\MT_{\Hs_\Z}(\Z)$, and let $D = D(G)$ be the weak mixed Mumford--Tate domain associated to $G$. Let $\varphi\colon X(\CC) \to \Lambda \backslash D$ be the period map of $\Hs_\Z$. Let $V \subseteq X \times \check{D}$ be an algebraic subvariety, where $\check{D}$ is the compact dual of $D$. Let $W$ be an irreducible analytic component of $V \cap \Lambda$ such that the projection of $W$ to $X$ is not contained in any proper weakly special subvariety. Then
  \[
    \codim_{X \times \check{D}}(W) \geq \codim_{X \times \check{D}}(V) + \codim_{X \times \check{D}}(\Lambda).
  \]
\end{conjecture}

To deduce Theorem \ref{thm:main} from Conjecture \ref{conj:ax-schanuel} (or from \cite[Thm.~1]{Ax72} in the $n = 1$ case), there are two main steps:
\begin{enumerate}
  \item Show that the complex unipotent Albanese map is (almost) the period map for a variation of mixed Hodge structure.
  \item Formally deduce properties of the $p$-adic unipotent Albanese map from corresponding properties of the complex map.
\end{enumerate}

Once this has been done, we use Corollary \ref{cor:main} to deduce various Diophantine consequences via the Chabauty--Kim method. Our main results for curves are in the following setting:
\begin{situation}
  \label{curve-situation}
  Let $X$ be a smooth, geometrically connected, hyperbolic algebraic curve over a number field $F$ such that Conjecture \ref{conj:ax-schanuel} holds for the canonical unipotent variations of mixed Hodge structure on $X$. Let $S$ be a finite set of primes of $F$ containing all primes of bad reduction for $X$. Suppose we are in one of the following settings:
  \begin{enumerate}
    \item The Fontaine--Mazur conjecture or the Bloch--Kato conjecture is true.
    \item $X = \proj^1 \setminus \{p_1, \dots, p_s\}$, where $s \geq 3$ and $p_1, \dots, p_s \in \OO_{F, S} \cup \{\infty\}$.
    \item $X$ is a CM elliptic curve minus the origin.
    \item There exists a dominant regular map $X_{\bar{F}} \to Y_{\bar{F}}$ for some smooth projective curve $Y$ over $F$ of genus $g \geq 2$ such that the Jacobian variety of $Y$ has CM over $\bar{F}$.
  \end{enumerate}
\end{situation}

\begin{theorem}
  \label{thm:curves}
  In Situation \ref{curve-situation}, let $V \subseteq \Res_\Q^F X$ be an irreducible, positive-dimensional closed $\Q$-subvariety. Let $\varphi\colon V' \to V$ be a surjective morphism of $\Q$-varieties such that $V'$ is smooth and $\varphi$ is birational. Suppose $S$ also contains all primes lying above primes of bad reduction for $V'$. Let $p$ be a rational prime that splits completely in $F$, such that $\pri \notin S$ for all $\pri$ lying above $p$. Let $\VV'$ be a smooth $S$-integral model of $V'$, let $b \in \VV'(\OO_{F, S})$ be a base point, and let $\UU \subseteq \VV'(\OO_{F_\pri})$ be the residue disk containing $b$. Let $j_{V', n}\colon \UU \to (\Pi_{V', n}^{\dR}/F^0)(\Q_p)$ be the unipotent Albanese map. Then for all sufficiently large $n$, the Chabauty--Kim locus
  \[
  \UU_{V', n} := j_{V', n}^{-1}(\log_p(H_f^1(\Gal(\bar{\Q}/\Q), \Pi_{V', n}))) \subseteq \UU
  \]
  is non-Zariski-dense in $V'$.
\end{theorem}

For $F = \Q$, the above result is already a theorem, due to \cite{Kim-siegel,Kim-alb,Kim-elliptic,CK,EH18} in the various cases of Situation \ref{curve-situation}.

\begin{corollary}
  \label{cor:finiteness}
  In Situation \ref{curve-situation}, let $\XX$ be a smooth $S$-integral model of $X$. Then $\XX(\OO_{F, S})$ is finite.
\end{corollary}
\begin{proof}
  Let $\Rcal$ be a smooth $S'$-integral model of $R = \Res_\Q^F X$, where $S'$ is the set of primes of $\Z$ lying below primes in $S$. Restriction of scalars is compatible with base change, so $\Rcal(\Z_{S'}) = \XX(\OO_{F, S})$. To prove finiteness of $\Rcal(\Z_{S'})$, it suffices to show that $\Rcal(\Z_{S'})$ is non-Zariski-dense in every irreducible, positive-dimensional closed $\Q$-subvariety $V \subseteq R$.

  By resolution of singularities, we can construct a surjective proper morphism of smooth $\Q$-varieties $\varphi\colon R' \to R$ such that $\varphi$ is birational, the strict transform $V' \subseteq R'$ of $V$ is smooth, and the exceptional locus of $\varphi$ is transverse to $V'$. Let $T$ be a finite set of rational primes such that $S' \subseteq T$, the varieties $R'$, $V'$, and $V$ and the map $\varphi$ are defined over $\Z_T$, and $R'$ and $V'$ have good reduction at every prime $\ell \notin T$. Let $\VV'$ be a smooth $T$-integral model of $V'$. Applying Theorem \ref{thm:curves} to each residue disk of $\VV'$, we see that $\VV'(\Z_T)$ is non-Zariski-dense in $V'$. Since $\varphi$ is proper and is an isomorphism away from an exceptional locus transverse to $V'$, it follows that $\Rcal(\Z_T) \cap V$ is non-Zariski-dense in $V$. But $\Z_{S'} \subseteq \Z_T$, so we are done.
\end{proof}

Dogra \cite{Dogra19} has independently proven unlikely intersection results closely related to those of this paper, using purely $p$-adic (as opposed to complex Hodge-theoretic) methods.

The structure of this paper is as follows: In \S\ref{section:period-domains}, we discuss mixed Hodge varieties, variations of mixed Hodge structure, and the Ax--Schanuel conjecture, and we describe the unipotent Albanese map over the complex numbers in terms of certain canonical variations of Hodge structure. In \S\ref{section:p-adic}, we transfer from the complex setting to the $p$-adic setting, using the results of \S2 to deduce Theorem \ref{thm:main}. In \S\ref{section:chabauty}, we note an implication for Chabauty's method of the Ax--Schanuel theorem for abelian varieties; this is the abelian analogue of the main results of this paper. In \S\ref{section:chabauty-kim}, we prove Theorem \ref{thm:curves} by making the necessary modifications to the proofs over $\Q$. Some possible directions of future work are discussed in \S\ref{section:future}.

\subsection*{Acknowledgements}
This work was supported in part by the Simons Collaboration on Arithmetic Geometry, Number Theory, and Computation. The author is grateful to Jordan Ellenberg, Daniel Erman, Bruno Klingler, Jacob Tsimerman, and Anthony V\'{a}rilly-Alvarado for helpful discussions related to this paper. The author also thanks Netan Dogra and an anonymous referee for their helpful comments on previous revisions of this paper.

\section{Higher Albanese manifolds as period domains}
\label{section:period-domains}
In this section, we show that the complex unipotent Albanese map is almost the period map for the canonical variation of Hodge structure, which assigns to each $x \in X$ the truncated path algebra $P_{b, x} := \Z\pi_1(X; b, x)/J^{n + 1}$ (where $J$ is the augmentation ideal). The primary references are \cite{HZ87}, \cite{BaTs17}, and \cite{Klingler17}.

Let $X$ be a smooth connected variety over $\CC$. Choose a point $b \in X(\CC)$. Let $\Pi$ be the unipotent completion of the fundamental group of $X$ with base point $b$, and let $\Pi_n$ be the quotient of $\Pi$ by the $(n + 1)$-st level of the lower central series. This is an algebraic group over $\CC$ that comes equipped with a natural mixed Hodge structure.

The $n$-th \emph{higher Albanese manifold} is the complex manifold
\[
  \Alb_n(X) := \pi_1(X, b) \backslash \Pi_n(\CC) / F^0 \Pi_n(\CC).
\]
This does not depend on $b$, and is typically not an algebraic variety for $n > 1$. (For $n = 1$, this is the classical Albanese variety.)

The \emph{unipotent Albanese map} is the map
\[
  \alpha_n\colon X(\CC) \to \Alb_n(X)
\]
defined via Chen's $\pi_1$-de Rham theory by mapping $x \in X(\CC)$ to the iterated integration functional $\omega \mapsto \int_{b}^{x} \omega$; the quotient by the topological fundamental group ensures this is independent of the choice of path from $b$ to $x$. This map does depend on $b$.

By \cite[Prop.~4.22]{HZ87}, the map $x \mapsto P_{b, x}$ defines a graded-polarizable variation of mixed Hodge structure $\Hs$ on $X(\CC)$, called the $n$-th \emph{canonical variation}. This variation is \emph{unipotent} in the sense that the variations of (pure) Hodge structure on the graded quotients for the weight filtration are constant.

Let $M_n$ be the generic Mumford--Tate group of the $n$-th canonical variation. The monodromy representation for the canonical variation is given by
\begin{align*}
  \rho\colon \pi_1(X, b) &\to \GL(P_{b, x}), \\
  \beta &\mapsto (\gamma \mapsto \gamma \beta^{-1}),
\end{align*}
the right regular representation of $\pi_1(X, b)$ on the $P_{b, b}$-bimodule $P_{b, x}$. Since the canonical variations are unipotent, the monodromy representation $\rho$ is also unipotent. Let $G_n$ be the Zariski closure of the image of the monodromy representation. Then $G_n$ is unipotent, so $\rho$ factors through the unipotent completion $\pi_1(X, b) \to \Pi$. The kernel of $\rho$ is $J^{n + 1} \cap \pi_1(X, b)$, so $G_n \cong \Pi_n$.

The \emph{weak mixed Mumford--Tate domain} $D = D(\Pi_n)$ associated to the monodromy group $\Pi_n$ is the $\Pi_n(\CC)$-orbit of the canonical mixed Hodge structure on $P_{b, b}$ in the full period domain of graded-polarized mixed Hodge structures on $P_{b, b}$. In particular, $D$ is a homogeneous space for $\Pi_n(\CC)$, and is a mixed Mumford--Tate domain in the sense of \cite[\S3.1]{Klingler17}.

As explained in \cite[\S5]{HZ87}, the stabilizer of a point in $D$ is $F^0 \Pi_n(\CC)$; thus, $D \cong \Pi_n(\CC)/F^0 \Pi_n(\CC)$. One quotient of $D$ by a discrete group is the higher Albanese manifold:
\[
\Alb_n(X) = \pi_1(X, b) \backslash D.
\]
Another such quotient is the connected mixed Hodge variety \cite[\S3]{Klingler17}
\[
  \Hod^0(X, \Hs, D) := \Hod_\Lambda^0(\Pi_n, D) = \Lambda \backslash D,
\]
where $\Lambda = (\Pi_n(\Q) \cap \GL(P_{b, b}))$. The action of $\pi_1(X, b)$ on $D$ factors through $\Pi_n(\Q) \cap \GL(P_{b, b})$, so there is a natural covering map
\[
\Alb_n(X) \to \Hod^0(X, \Hs).
\]
By \cite[Cor.~5.20(i)]{HZ87}, the period map
\[
\varphi_n\colon X(\CC) \to \Hod^0(X, \Hs)
\]
factors through this covering map via the unipotent Albanese map $\alpha_n$. (Hain and Zucker are using a different classifying space as the target of the period map, but the conclusion is the same since one just has to show that $\varphi_n(x)$ depends only on $\alpha_n(x)$.)

Now we can immediately deduce:
\begin{lemma}
  \label{lemma:higher-albanese}
  Let $X$ be a smooth connected variety, and fix a point $b \in X(\CC)$ and an integer $n \geq 1$. If $n > 1$, assume Conjecture \ref{conj:ax-schanuel} for the $n$-th canonical variation on $X$. Let $\alpha_n\colon X \to \Alb_n(X)$ be the unipotent Albanese map with base point $b$. Let $D = \Pi_n(\CC)/F^0 \Pi_n(\CC)$. Let $V \subseteq X \times \Pi_n/F^0 \Pi_n$ be an algebraic subvariety. Let $\Gamma \subseteq X(\CC) \times D$ be the graph of $\alpha_n$ pulled back via the covering map $D \to \Alb_n(X)$. Let $W$ be an irreducible analytic component of $V \cap \Gamma$ such that the projection of $W$ to $X$ is not contained in any proper weakly special subvariety. Then
  \[
    \codim_{X \times D}(W) \geq \codim_{X \times D}(V) + \codim_{X \times D}(\Gamma).
  \]
\end{lemma}
\begin{proof}
  Applying Conjecture \ref{conj:ax-schanuel} (or \cite[Thm.~1]{Ax72} if $n = 1$), we obtain the lemma where the higher Albanese manifold is replaced with the mixed Hodge variety $\Hod^0(X, \Hs)$ associated to the $n$-th canonical variation of Hodge structure. Since $\Pi_n(\CC)/F^0 \Pi_n(\CC)$ is compatibly a covering space of both $\Alb_n(X)$ and $\Hod^0(X, \Hs)$, the graph $\Gamma$ is the same in both cases, so this is in fact a special case of the Ax--Schanuel conjecture.
\end{proof}

\section{\texorpdfstring{$p$}{p}-adic Ax--Schanuel for the unipotent Albanese map}
\label{section:p-adic}
In this section, we deduce Theorem \ref{thm:main} from Lemma \ref{lemma:higher-albanese} by transferring from the complex setting to the $p$-adic setting. This argument closely follows (and was partly inspired by) Lawrence and Venkatesh's application of Ax--Schanuel for a pure variation of Hodge structure \cite[\S9]{Lawrence-Venkatesh}, which they also use to deduce Diophantine consequences (in their case for integral points in certain families of varieties).

Let $X$ be a smooth, connected, positive-dimensional algebraic variety of good reduction over a nonarchimedean local field $K/\Q_p$, and let $\XX/\OO_K$ be a smooth integral model. Fix a base point $b \in \XX(\OO_K)$, and let $\UU \subseteq \XX(\OO_K)$ be the residue disk containing $b$. Let $\Pi^{\dR}$ be the de Rham fundamental group of $X$ with base point $b$, and let $\Pi_n^{\dR}$ be the quotient of $\Pi^{\dR}$ by the $(n + 1)$-st level of the lower central series. Recall from the introduction the $p$-adic unipotent Albanese map
\[
j_n\colon \UU \to (\Pi_n^{\dR}/F^0)(K).
\]
Let $V \subseteq X \times \Pi_n^{\dR}/F^0$ be an algebraic subvariety, and let $\Gamma \subseteq \UU \times (\Pi_n^{\dR}/F^0)(K)$ be the graph of $j_n$. Let $W$ be an irreducible analytic component of $V(K) \cap \Gamma$, and suppose
\begin{equation}
  \label{eqn:codim-assumption}
  \codim_{X \times \Pi_n^{\dR}/F^0}(W) < \codim_{X \times \Pi_n^{\dR}/F^0}(V) + \codim_{X \times \Pi_n^{\dR}/F^0}(\Gamma).
\end{equation}
To prove Theorem \ref{thm:main}, we must show that the projection $\pr_X(W)$ of $W$ to $X$ is contained in a proper weakly special subvariety of $X$.

We work locally around an arbitrary point $w \in V(K) \cap \Gamma$. The $K$-analytic local ring of $\XX(\OO_K) \times (\Pi_n^{\dR}/F^0)(K)$ at $w$ is isomorphic to some Tate algebra $R = K\langle x_1, \dots, x_N \rangle$, and $\Gamma$ is defined locally near $w$ by $G_1 = \dots = G_r = 0$ for some $G_1, \dots, G_r \in R$.

The coordinates of the map $j_n$ (and hence the equations defining its graph $\Gamma$) are given by the $p$-adic iterated integration functionals $x \mapsto \int_b^x \omega$ of depth $n$ (where $\omega = \omega_1 \dots \omega_n$ is given by the choice of coordinate of the affine space $\Pi_n^{\dR}/F^0$). These iterated integration functionals are defined locally by the same differential equations in both the $p$-adic and complex analytic settings; hence, when the power series $G_1, \dots, G_r$ are viewed as complex power series (via an embedding $K \inject \CC$), they converge in a small ball $U_\CC$ around $w$ in $X(\CC)$, and they locally define the graph $\Gamma'$ of (a local lifting of) the holomorphic unipotent Albanese map
\[
\alpha_n \colon U_\CC \to (\Pi_n^{\dR}/F^0)(\CC) = D,
\]
which is also given by iterated integration (and hence by the same differential equations). Here, $D$ is the weak mixed Mumford--Tate domain from \S\ref{section:period-domains}.

Let $I$ be the defining ideal of $V$ as an algebraic subvariety of $X \times \Pi_n^{\dR}/F^0$. Choose functions $F_1, \dots, F_s \in I$ such that the equations $F_1 = \dots = F_s = G_1 = \dots = G_r = 0$, where $s + r = \codim(V(K) \cap \Gamma)$, locally define an analytic variety of the same dimension as $V(K) \cap \Gamma$ at $w$. (This is to avoid complications if $V(K) \cap \Gamma$ happens not to be an analytic local complete intersection at $w$.) By the assumption on the codimension \eqref{eqn:codim-assumption}, $s < \codim V$.

Since $F_1, \dots, F_s$ are regular functions, we can view them over $\CC$ without convergence issues; let $V' = \{F_1 = \dots = F_s = 0\}$. Consider any irreducible component $W'$ of $V' \cap \Gamma'$. By construction, $V' \cap \Gamma'$ is locally a complete intersection at $w$, so
\[
  \codim_{X \times D}(W') = s + r < \codim_{X \times D}(V) + r = \codim_{X \times D}(V') + \codim_{X \times D}(\Gamma').
\]
By Lemma \ref{lemma:higher-albanese}, $\pr_X(W')$ is contained in a proper weakly special subvariety of $X$. Since $V \subseteq V'$, it follows that $\pr_X(V(\CC) \cap \Gamma')$ is contained in a finite union of such varieties, and using the fixed embedding $K \inject \CC$, we obtain the same for $\pr_X(V(K) \cap \Gamma)$.
\qed

\section{Implications for Chabauty's method}
\label{section:chabauty}
In the case $n = 1$, motivated by Siksek's heuristic argument \cite[\S2]{Siksek13}, we note the following consequence of the Ax--Schanuel theorem for abelian varieties \cite[Thm.~1]{Ax72}:
\begin{proposition}
  \label{prop:n=1}
  Let $V$ be a smooth, projective, geometrically connected variety over $\Q$. Let $A$ be the Albanese variety of $V$. Let $d = \dim V$, let $g = \dim A$, and let $r$ be the rank of $A(\Q)$. Let $\iota\colon V \to A$ be the Abel--Jacobi map associated to a chosen base point $b \in V(\Q)$. Let $p$ be a prime of good reduction for $V$. Let $r + \delta$ be the $\Z_p$-rank of the Bloch--Kato Selmer group $H_f^1(\Gal(\bar{\Q}/\Q), T_p A)$. (Note that $\delta = 0$ if $\Sha_A[p^\infty]$ is finite, as is conjectured.) 

  Suppose $r + \delta \leq g - d$. Then
  \[
  \iota(V(\Q_p)) \cap \overline{A(\Q)},
  \]
  where the closure is in the $p$-adic analytic topology in $A(\Q_p)$, is contained in a finite union of cosets of proper abelian subvarieties of $A$.
\end{proposition}
\begin{proof}
  We have a commutative diagram
  \[
  \begin{tikzcd}
    V(\Q) \ar[r] \ar[d,"\iota"] & V(\Q_p) \ar[d] \\
    A(\Q) \otimes_{\Z} \Q_p \ar[r] \ar[d] & \Lie(A) \otimes \Q_p \ar[d] \\
    H_f^1(\Gal(\bar{\Q}/\Q), \Pi_1) \ar[r,"\log_p"] & (\Pi_1^{\dR}/F^0)(\Q_p)
  \end{tikzcd}
  \]
  where $\Pi_1 = T_p A \otimes_{\Z_p} \Q_p$, and where $\log_p$ is an algebraic map of $\Q_p$-varieties. We have
  \[
  \dim \Pi_1^{\dR}/F^0 = \dim A = g.
  \]
  Also,
  \[
  \dim H_f^1(\Gal(\bar{\Q}/\Q), \Pi_1) = \rank_{\Z_p} H_f^1(\Gal(\bar{\Q}/\Q), T_p A) = r + \delta.
  \]

  Suppose $r + \delta \leq g - d$. Then the image of $\log_p$ is an algebraic subvariety of $\Pi_1^{\dR}/F^0$ of codimension at least $d$. By Corollary \ref{cor:main} (which, recall, is known unconditionally in the case $n = 1$), there is a finite subset $S \subseteq V(\Q_p)$ such that $\iota^{-1}(A(\Q)) \setminus S$ is contained in a finite union of proper weakly special subvarieties of $V$.

  By comparison with the complex setting, weakly special subvarieties for the Albanese map are exactly the pullbacks to $V$ of bialgebraic subvarieties for the uniformization map of $A$, which are exactly the cosets of abelian subvarieties, completing the proof.
\end{proof}

\begin{remark}
  \begin{enumerate}
    \item If the condition $r + \delta \leq g - d$ remains true upon restriction to any abelian subvariety of $A$, then applying Proposition \ref{prop:n=1} inductively implies $\iota(V(\Q_p)) \cap \overline{A(\Q)}$, and hence $V(\Q)$, are finite. However, this does not always happen; see \cite[\S2.2]{Dogra19} for a counterexample, which demonstrates Siksek's heuristic does not always work and refutes the conjecture made in the author's thesis \cite[Conj.\ 5.1]{Hast18}.
    \item The motivating case is when $V = \Res_{\Q}^{F} X$ is the restriction of scalars of a curve $X/F$ with $F$ a number field. Let $g_X$ be the genus of $X$, and $J$ the Jacobian variety of $X$. Then $r$ is the rank of $J(F)$, $d = [F : \Q]$, and $g - d = d(g_X - 1)$.
    \item This is essentially proven in \cite[Cor.\ 2.1]{Dogra19}, which Dogra proved independently while this paper was in preparation. (Dogra doesn't explicitly state in Cor.\ 2.1 that the Chabauty locus is contained in a translate of a proper abelian subvariety, but that can be extracted from the proof.)
  \end{enumerate}
\end{remark}

\section{Implications for the Chabauty--Kim method}
\label{section:chabauty-kim}
In this section, we deduce Theorem \ref{thm:curves} from Corollary \ref{cor:main}. The following lemma is the key intermediate result:
\begin{lemma}
  \label{lemma:DH}
  In Situation \ref{curve-situation}, let $V \subseteq \Res_{\Q}^{F} X$ be an irreducible, positive-dimensional closed subvariety. Let $\varphi\colon V' \to V$ be a surjective morphism of $\Q$-varieties such that $V'$ is smooth and $\varphi$ is birational. Expand $S$ (if necessary) to include all primes of bad reduction for $V'$. Let $p$ be a rational prime that splits completely in $F/\Q$ such that $\pri \notin S$ for all $\pri$ lying above $p$. Let $\Pi_{V', n}$ be the depth $n$ unipotent fundamental group of $V'$. Then
  \[
  \lim_{n \to \infty} \codim_{\Pi_{V', n}^{\dR}/F^0} \log_p(H_f^1(\Gal(\bar{\Q}/\Q), \Pi_{V', n})) = \infty.
  \]
\end{lemma}

In other words, the codimension of the image of the global Selmer variety inside the local Selmer variety grows without bound as $n \to \infty$.

Before proving the lemma, we prove Theorem \ref{thm:curves} assuming the lemma.

\begin{proof}[Proof of Theorem \ref{thm:curves}]
  Recall that the goal is to prove that the set
  \[
  \UU_{V', n} := j_{V', n}^{-1}(\log_p(H_f^1(\Gal(\bar{\Q}/\Q), \Pi_{V', n}))) \subseteq \UU,
  \]
  where $\UU \subseteq \VV'(\OO_{F_\pri})$ is the residue disk containing $b$, is non-Zariski-dense in $V'$.

  Adopting the notation of Lemma \ref{lemma:DH}, let $Z$ be the image of the regular map
  \[
  \log_p\colon H_f^1(\Gal(\bar{\Q}/\Q), \Pi_{V', n}) \to \Pi_{V', n}^{\dR}/F^0.
  \]
  By Lemma \ref{lemma:DH}, if $n$ is sufficiently large, then $Z$ has codimension at least $\dim V'$. By Corollary \ref{cor:main}, it follows that $j_{V',n}^{-1}(Z)$ is contained in the union of a finite set and finitely many weakly special subvarieties of $V'$. In particular, $j_{V',n}^{-1}(Z)$ is non-Zariski-dense in $V'$.
\end{proof}

\begin{remark}
  Using the method of \cite[\S5]{Dogra19}, one can show additionally that the weakly special subvarieties arising above are defined over a number field, hence (by induction on dimension) that $\UU_n$ itself is finite. Since this is not needed to prove finiteness of $\XX(\OO_{F, S})$, we do not do so here.
\end{remark}

In the remainder of this section, we verify Lemma \ref{lemma:DH} in each of the cases of Situation \ref{curve-situation}.

\subsection{Generalities}
\label{subsection:chabauty-kim-generalities}

Let $X$ be a smooth, geometrically connected, hyperbolic algebraic curve over a number field $F$. Let $R = \Res_{\Q}^{F} X$, and let $V \subseteq R$ be an irreducible, positive-dimensional closed subvariety. Let $\varphi\colon V' \to V$ be a surjective regular map of $\Q$-varieties such that $V'$ is smooth and $\varphi$ is birational.

Let $S$ be a finite set of primes of $\Q$ containing all of the primes of bad reduction for $X$ and for $V'$. Let $p \notin S$ be a rational prime that splits completely in $F/\Q$. Let $T = S \cup \{p\}$. Let $G_T = \Gal(\Q_T/\Q)$, where $\Q_T$ is the maximal algebraic extension of $\Q$ unramified outside $T$. Let $\Pi_{V', n}$ be the depth $n$ unipotent fundamental group of $V'$.

We will prove that, in each of the cases of Situation \ref{curve-situation}, for a suitable quotient $U_n$ of $\Pi_{V', n}$, we have the following inequality of dimensions: There exists $c < 1$ such that
\[
\dim H_f^1(G_T, U_n) < c \cdot \dim U_n^{\dR}/F^0 U_n^{\dR}.
\]
This immediately implies the conclusion of Lemma \ref{lemma:DH}.

In each case, the quotient $U_n$ will in fact be a quotient of the algebraic group
\[
\Upsilon_n := \Upsilon/(\Upsilon \cap \Pi_R^{n + 1}),
\]
where $\Upsilon$ is the image of the induced morphism $\Pi_{V'} \to \Pi_R$ of pro-unipotent fundamental groups, and the superscript denotes the level of the lower central series of $\Pi_R$. (This is the construction of \cite[\S3.3]{EH18}, except here we start with the full pro-unipotent fundamental groups, not the metabelian quotients.) Note that, by construction, $\Pi_{V', n}$ surjects onto $\Upsilon_n$, and $\Upsilon_n$ is an algebraic subgroup of $\Pi_{R, n}$.

In this subsection, before splitting into the different cases, we state some initial facts that apply in general. These reductions closely follow those of \cite{Kim-alb} and \cite{CK}.

First, we have an inequality
\[
\dim H_f^1(G_T, U_n) \leq \dim H^1(G_T, U_n),
\]
so it suffices to bound the latter dimension. We also have short exact sequences
\[
0 \to Z_n \to U_n \to U_{n - 1} \to 0,
\]
where $Z_n = U^n/U^{n + 1}$ is the $n$-th graded piece with respect to a filtration on $U$ (for example, the filtration induced by the lower central series filtration on $\Pi_R$). Applying cohomology, we obtain an exact sequence
\[
H^0(G_T, U_{n - 1}) \to H^1(G_T, Z_n) \to H^1(G_T, U_n) \to H^1(G_T, U_{n - 1}).
\]
By a weight argument, $H^0(G_T, U_{n - 1}) = 0$, so it follows that
\[
\dim H^1(G_T, U_N) \leq \sum_{n=1}^{N} \dim H^1(G_T, Z_n).
\]

Now that we have reduced to Galois cohomology with abelian coefficients, we can use standard tools of Galois cohomology. By the global Poincar\'e--Euler characteristic formula,
\[
\dim H^1(G_T, Z_n) = \dim H^0(G_T, Z_n) + \dim H^2(G_T, Z_n) + \dim Z_n - \dim Z_n^{+},
\]
where $Z_n^{+}$ is the subspace fixed by the action of complex conjugation on $Z_n$. Again by a weight argument, we have $\dim H^0(G_T, Z_n) = 0$.

On the other side, we have
\[
  \dim U_N^{\dR}/F^0 U_N^{\dR} = \dim U_N^{\dR} - \dim F^0 U_N^{\dR} = \sum_{n=1}^{N} \dim Z_n - \dim F^0 U_N^{\dR}
\]
So, to prove Lemma \ref{lemma:DH}, it suffices to show the following asymptotic inequalities:
\begin{enumerate}
  \item $\sum_{n=1}^{N} \dim Z_n^{+} \geq c \cdot \sum_{n=1}^{N} \dim Z_n$ for some $c > 0$;
  \item $\dim F^0 U_n^{\dR} \ll \dim U_n$ for $n \gg 1$; and
  \item $\dim H^2(G_T, Z_n) \ll \dim Z_n$ for $n \gg 1$.
\end{enumerate}
The first two turn out to be straightforward in general, while the bound on $H^2$ is difficult.

\subsection{Arbitrary hyperbolic curves, conditionally}
\label{subsection:hyperbolic-curves}
In \cite{Kim-alb}, Kim proves the dimension hypothesis for an arbitrary curve $X$ over $\Q$ of genus $g \geq 2$, conditional on either the Bloch--Kato conjecture or the Fontaine--Mazur conjecture. We generalize this to varieties $V' \onto V \subseteq R = \Res_{\Q}^{F} X$ as above.

In this section, let $U_n = \Upsilon_n$, as defined in \S\ref{subsection:chabauty-kim-generalities}, and let $Z_n = (\Upsilon \cap \Pi_R^n)/(\Upsilon \cap \Pi_R^{n + 1})$, where the superscripts denote the lower central series.

Let $Z_{X, n} = \Pi_X^n/\Pi_X^{n + 1}$ be the $n$-th graded piece of $\Pi_X$, and likewise for $Z_{R, n}$. Since $R_F \cong X_F^{\times [F : \Q]}$, we have a dominant morphism $V'_{F} \to X_F^{\times d}$, where $d = \dim V'$. By \cite[Lemma 4.1]{EH18}, this induces a surjection of pro-unipotent fundamental groups, hence a $\Gal(\bar{F}/F)$-equivariant surjection $Z_n \onto Z_{X, n}^{\oplus d}$, where $Z_{X, n} = \Pi_X^n/\Pi_X^{n+1}$ is the $n$-th graded piece of $\Pi_X$.

We also have a $\Gal(\bar{\Q}/\Q)$-equivariant injection $Z_n \inject Z_{R, n}$ (since by construction $Z_n$ is a subset of $Z_{R, n}$), and a $\Gal(\bar{F}/F)$-equivariant isomorphism $Z_R \cong Z_{X, n}^{\oplus [F : \Q]}$. Putting these together yields dimension bounds
\[
d \cdot \dim Z_{X, n} \leq \dim Z_n \leq [F : \Q] \cdot \dim Z_{X, n}.
\]

Let $b_1 = \dim Z_{X, 1}$, so $b_1 = 2g$ if $X$ is compact and $b_1 = 2g + s - 1$ if $X$ is non-compact with $s$ punctures. By \cite[\S3]{Kim-alb}, in the limit as $n \to \infty$,
\[
  \dim Z_{X, n} = \frac{b_1^n}{n} + o(b_1^n)
\]
if $X$ is non-compact and
\[
  \dim Z_{X, n} = \frac{(g + \sqrt{g^2 - 1})^n}{n} + o(b_1^n)
\]
if $X$ is compact, while
\[
  \dim F^0 \Pi_{X, n}^{\dR} \leq g^n,
\]
and, assuming the Fontaine--Mazur conjecture or the Bloch--Kato conjecture,
\[
  \dim H^2(G_{F, T}, Z_{X, n}) \leq P(n) g^n
\]
for some polynomial $P(n)$.

Since $Z_n$ has weight $n$, by comparison with complex Hodge theory, we have $\dim Z_n^{+} = \frac{1}{2} \dim Z_n$ for $n$ odd. The above dimension bounds on $Z_n$ and $Z_{X, n}$ then yield
\[
\sum_{n=1}^{N} \dim Z_n^+ \geq \sum_{\substack{n \leq N \\ n \text{ odd}}} \dim Z_n^+ = \frac{1}{2} \sum_{\substack{n \leq N \\ n \text{ odd}}} \dim Z_n \geq c \sum_{n=1}^{N} \dim Z_n
\]
for some constant $c > 0$, as was to be shown.

Next, the injection $U_n^{\dR} \inject \Pi_{R, n}^{\dR} \cong (\Pi_{X, n}^{\dR})^{\oplus [F : \Q]}$ is compatible with the Hodge structures, so
\[
  \dim F^0 U_n^{\dR} \leq [F : \Q] \cdot \dim F^0 \Pi_{X, n}^{\dR} = O(g^n).
\]
Since $g < b_1$, it follows that $F^0 U_n^{\dR}$ does not contribute to the asymptotic.

We can compare $H^2(G_T, Z_n)$ to $H^2(G_T, Z_{X, n})$ using the idea of \cite[Lemma 6.4]{EH18}: The corestriction map
\[
  H^2(G_{F, T}, Z_n) \to H^2(G_T, Z_n)
\]
is surjective because $Z_n$ is a divisible abelian group. By the semisimplicity theorem of Faltings and Tate, $Z_{X, 1}$ is a semisimple $G_{F, T}$-representation, so $(Z_{X, 1}^{\otimes n})^{\oplus [F : \Q]}$ is semisimple. Hence, the surjection $Z_{X, 1}^{\otimes n} \onto Z_{X, n}$ splits and realizes $Z_{X, n}$ as a direct summand, so $Z_{X, n}$ is also semisimple.

The $\Gal(\bar{F}/F)$-equivariant injection $Z_n \inject Z_{X, n}^{\oplus [F : \Q]}$ therefore realizes $Z_n$ as a direct summand. Cohomology preserves direct summands, so
\[
  \dim H^2(G_T, Z_n) \leq \dim H^2(G_{F, T}, Z_n) \leq [F : \Q] \cdot \dim H^2(G_{F, T}, Z_{X, n}) \leq P(n) g^n.
\]
Thus $H^2(G_T, Z_n)$ also does not contribute to the asymptotic.

This completes the proof of Lemma \ref{lemma:DH} in case (1) of Situation \ref{curve-situation}.
\qed

\subsection{Genus zero}
In this section, suppose $X = \proj^1 \setminus \{p_1, \dots, p_s\}$ with $s \geq 3$ and $p_1, \dots, p_s \in \OO_{F, S} \cup \{\infty\}$. Then $Z_{X, 1} = \Q_p(1)^{s - 1}$, and more generally, $Z_{X, n} = \Q_p(n)^{r_n}$ for some $r_n > 0$.

As explained in \cite{Kim-siegel} and \cite[\S8]{Hadian11}, we have $F^0 \Pi_X = 0$. Moreover, by the Soul\'e vanishing theorem, $H^1(G_T, \Q_p(2n)) = 0$ and $\dim H^1(G_T, \Q_p(2n + 1)) = 1$ for any $n \geq 1$. So
\[
\dim H_f^1(G_T, \Pi_{X, n}) \leq \sum_{i=1}^{n} \dim H^1(G_T, Z_{X, i}) = R + r_3 + r_5 + \dots + r_{2\lfloor n/2 \rfloor},
\]
where $R = \dim H^1(G_T, \Q_p(1))$. Since each $r_i$ is positive, for sufficiently large $n$, this is less than
\[
\dim \Pi_{X, n}/F^0 = \dim \Pi_{X, n} = (s - 1) + r_2 + \dots + r_n.
\]

To generalize to our setting, let $U_n = \Upsilon_n$ be the same quotient of $\Pi_{V', n}$ defined in \S\ref{subsection:chabauty-kim-generalities}. Then, as before, we have a $\Gal(\bar{\Q}/\Q)$-equivariant injection $U_n \inject \Pi_{R, n}$. Choosing an embedding $X \inject \aff_F^1$ as an open subscheme, functoriality of restriction of scalars gives an open embedding $R \inject \aff_{\Q}^{[F : \Q]}$, so $R$ is a rational variety.

Hence, by \cite[Prop.\ 4.15]{DG05}, $\Pi_R$ is mixed Tate, and thus so is $U_n$. So we can apply the same Soul\'e vanishing argument as above to see that $\sum_{i=1}^{n} \dim H^1(G_T, Z_i) \ll \dim U_n$. This completes the proof of Lemma \ref{lemma:DH} in case (2) of Situation \ref{curve-situation}.
\qed

\subsection{Punctured CM elliptic curves}
In this section, suppose $X = E \setminus \{O\}$, where $E$ is a CM elliptic curve over $F$, and $O \in E$ is the identity. We generalize Kim's results \cite[Thm.\ 0.1]{Kim-elliptic} to the setting of Lemma \ref{lemma:DH}.

Choose $p$ to split as $p = \pi \bar{\pi}$ in the CM field $K$. As described in \cite[\S1]{Kim-elliptic}, this gives a splitting
\[
  \Pi_{X, 1} \cong V_\pi(E) \oplus V_{\bar{\pi}}(E)
\]
into the rational $\pi$-adic and $\bar{\pi}$-adic Tate modules, which are one-dimensional and we take to be generated by elements $e$ and $f$, respectively. The Lie algebra $L$ of $\Pi_X$ is the pro-nilpotent completion of the free Lie algebra on generators $e$ and $f$, and $L$ has a Hilbert basis of Lie monomials in $e$ and $f$. 

As Kim shows, the Lie ideals $L_{\geq n, \geq n}$ generated by monomials of degree at least $n$ in each generator are Galois-equivariant. Let $W$ be the quotient of $\Pi_X$ corresponding to the quotient Lie algebra $L/L_{\geq 2, \geq 2}$. Let $W_n = W/W^{n+1}$ and $Z_{X, n} = W^n/W^{n + 1}$, where $W^\bullet$ denotes the lower central series filtration of $W$. Then we have a $\Gal(\bar{\Q}/\Q)$-equivariant isomorphism
\[
  Z_{X, n} \cong \Q_p(\chi^{n - 2}(1)) \oplus \Q_p(\bar{\chi}^{n - 2}(1)),
\]
where $\chi$ and $\bar{\chi}$ are the Galois characters associated to the Galois actions on $V_\pi(E)$ and $V_{\bar{\pi}}(E)$, respectively.

Let $U_n$ be the image of $\Upsilon_n$ in $W_n$, and let $Z_n$ be the image of $(\Upsilon \cap \Pi_R^n)/(\Upsilon \cap \Pi_R^{n+1})$ in $Z_{X,n}$. As explained in \cite[\S3]{Kim-elliptic}, we have $\dim F^0 Z_{X, 1}^{\dR} = \dim F^0 Z_{X, 2}^{\dR} = 1$, while $F^0 Z_{X, n}^{\dR} = 0$ for all $n \geq 3$, so the $F^0$ term does not contribute to the asymptotic. Also, by \cite[Claim 3.1]{Kim-elliptic}, $H^2(G_T, Z_{X, n}) = 0$ for all sufficiently large $n$. By the argument of Subsection \ref{subsection:hyperbolic-curves}, it follows that $H^2(G_T, Z_n) = 0$ as well. (Since $E$ has CM, we also do not need to appeal to the semisimplicity theorem, since $\Pi_{X, 1}$ is a direct sum of Galois characters.)

In this setting, $\dim Z_n$ is not growing: $\dim Z_{X, n} = 2$ for all $n$, so $2d \leq \dim Z_n \leq 2 [F : \Q]$. So, what we need to show is that $\dim Z_n^{+} < \dim Z_n$ for infinitely many $n$. But this follows by the same argument as before: by comparison with complex Hodge theory, $\dim Z_n^{+} = \frac{1}{2} \dim Z_n$ for $n$ odd. This completes the proof of Lemma \ref{lemma:DH} in case (3) of Situation \ref{curve-situation}.
\qed

\subsection{Curves dominating a curve with CM Jacobian}
In \cite{EH18}, the dimension hypothesis is proved for a smooth projective hyperbolic curve $X$ over $\Q$ such that there exists a smooth projective hyperbolic curve $Y/\Q$ with CM Jacobian and a dominant map $f\colon X_{\bar{\Q}} \to Y_{\bar{\Q}}$ (corresponding to case (4) of Situation \ref{curve-situation}). We now verify Lemma \ref{lemma:DH} in this setting.

In this section, $U_n$ is the following quotient of $\Upsilon_n$: Let $F'/F$ be a Galois extension such that $f$ is defined over $F'$, and let $R_Y = \Res_{\Q}^{F'}(Y)$. For any algebraic group $W$, let $\Psi_W = \Pi_W/\Pi_W^{(3)}$ be the metabelian quotient (where the superscript in parentheses denotes the derived series). Let $U$ be the image of $\Pi_{V'}$ in $\Psi_{R_Y}$, and let $U_n = U/(U \cap \Psi_{R_Y}^{n+1})$.

Let $Z_{X, n} = \Psi_Y^n/\Psi_Y^{n + 1}$ be the $n$-th graded piece of $\Psi_Y$, and likewise for $Z_{R_Y, n}$. Let $Z_n = Z_{R_Y, n} \cap U_n$. Since $(R_Y)_{F'} \cong Y_{F'}^{\times [F' : \Q]}$, we have a dominant morphism $V'_{F'} \to Y_{F'}^{\times d}$, where $d = \dim V'$. By \cite[Lemma 4.1]{EH18}, this induces a surjection of pro-unipotent fundamental groups, hence a $\Gal(\bar{F'}/F')$-equivariant surjection $Z_n \onto Z_{Y, n}^d$, where $Z_{Y, n} = \Psi_Y^n/\Psi_Y^{n+1}$ is the $n$-th graded piece of $\Psi_Y$.

We also have a $\Gal(\bar{\Q}/\Q)$-equivariant injection $Z_n \inject Z_{R_Y, n}$, and a $\Gal(\bar{F'}/F')$-equivariant isomorphism $Z_{R_Y} \cong Z_{Y, n}^{\oplus [F' : \Q]}$. Putting these together yields dimension bounds
\[
d \cdot \dim Z_{Y, n} \leq \dim Z_n \leq [F' : \Q] \cdot \dim Z_{Y, n}.
\]
Note that, as computed in \cite{CK}, we have $\dim Z_{Y, n} \sim An^{2g - 1}$ for some constant $A > 0$ in the limit as $n \to \infty$.

We obtain $F^0 Z_n^{\dR} = O(n^g)$ exactly as in \cite[Lemma 6.1]{EH18}; note that since the Hodge filtration is compatible with extension of the base field, it doesn't actually matter whether $p$ splits completely in $F'$.

Likewise, just as in \cite[Lemma 6.4]{EH18}, we have $\dim H^2(G_T, Z_n) = O(n^{2g - 2})$. To briefly summarize the argument: Since the Jacobian of $Y$ has CM, the Tate module $Z_{Y, 1}$ splits as a direct sum of characters, so the surjection $Z_{Y, 1}^{\otimes n} \onto Z_{Y, n}$ splits, implying $Z_{Y, n}$ is also semisimple. Thus the inclusion $Z_n \inject Z_{Y, n}^{\oplus [F' : \Q]}$ realizes $Z_n$ as a direct summand. Cohomology preserves direct summands, so
\[
  \dim H^2(G_T, Z_n) \leq [F' : \Q] \dim H^2(G_{F', T}, Z_{Y, n}) = O(n^{2g - 2}).
\]

Finally, it remains to show that $\dim Z_n^{+} \geq c \cdot \dim Z_n$ for some constant $c > 0$. This follows from a minor modification of \cite[\S7]{EH18}: We pick a $\Q_p$-basis of eigenvectors for the action of complex conjugation on $Z_{V', 1}$, project this down to a $\Q_p$-basis of $Z_1$, and then use the combinatorial argument of \cite[Lemma 7.3]{EH18} to construct a sufficiently large subspace invariant under complex conjugation. The only modification is that we are using the surjection $U_n \onto U_{Y, n}^{\oplus d}$, so the argument is carried out for the direct sum of $d$ copies of $Z_{Y, n}$. This just multiplies the dimensions by a factor of $d$, which poses no problem for our application.

Putting these together, we have proven Lemma \ref{lemma:DH} in case (4) of Situation \ref{curve-situation}.
\qed

\section{Future directions}
\label{section:future}
Assuming the Ax--Schanuel conjecture for variations of mixed Hodge structure, the above method can be used to prove Lang's conjecture on non-Zariski-density of rational points for any variety $V$ satisfying the asymptotic dimension hypothesis
\[
\lim_{n \to \infty} \codim_{\Pi_{V, n}^{\dR}/F^0} \log_p(H_f^1(\Gal(\bar{\Q}/\Q), \Pi_{V, n})) = \infty.
\]
A natural question is thus which varieties of general type satisfy this condition.

Unfortunately, aside from the case of curves, many varieties of general type are simply-connected, in which case $\Pi_{V, n}$ is trivial and the dimension hypothesis cannot hold. It is an interesting problem to find classes of varieties of general type for which the dimension hypothesis does hold, but Lang's conjecture is not already known. One useful test case may be varieties of the form $\proj^n \setminus D$, where $D$ is an effective divisor; the Selmer varieties may be amenable to a more explicit description in such cases.

Another problem is to make this higher-dimensional method algorithmic. An interesting test case for this would be restrictions of scalars of $\proj^1 \setminus \{0, 1, \infty\}$ over a number field. Such an algorithm would require both explicit computation of the special subvarieties occurring in the recursive construction in the proof of Theorem \ref{thm:curves}, and a method for computing the finite ``exceptional sets'' at each stage.
\printbibliography

\end{document}